\documentclass[12pt,a4paper]{article}
\pdfoutput=1

\usepackage[T1]{fontenc}
\usepackage[utf8]{inputenc}
\usepackage[UKenglish]{isodate}
\cleanlookdateon

\usepackage{amsfonts,amsmath,amssymb,amsthm,mathtools,dsfont,microtype} 
\usepackage[usenames,dvipsnames]{xcolor} 
\usepackage[noBBpl]{mathpazo} 

\usepackage[labelsep=period,labelfont=bf,justification=centering]{caption}
\usepackage{float,graphicx,subcaption}
\usepackage{booktabs,makecell}
\usepackage{enumitem,mdwlist}
\setlist[itemize]{topsep=0ex,itemsep=0ex,parsep=0.4ex}
\setlist[enumerate]{topsep=0ex,itemsep=0ex,parsep=0.4ex}

\usepackage{parskip,fullpage}
\usepackage{standalone}
\usepackage{thm-restate}

\usepackage[hyphens]{url} 
\usepackage[linktoc=all,hidelinks,colorlinks,unicode=true]{hyperref} 
\usepackage[capitalise,compress,nameinlink,noabbrev]{cleveref} 
\hypersetup{linkcolor={blue!70!black},citecolor={black},urlcolor={blue!70!black}}

\usepackage{tikz}

\newtheorem{theorem}[subsection]{Theorem}
\newtheorem{lemma}[theorem]{Lemma}
\newtheorem{claim}[theorem]{Claim}
\newtheorem*{claim*}{Claim}
\newtheorem{conjecture}[theorem]{Conjecture}

\newtheorem{question}[theorem]{Question}
\theoremstyle{definition}

\newenvironment{poc}{\begin{proof}}{\end{proof}}

\renewcommand{\ge}{\geqslant}
\renewcommand{\le}{\leqslant}
\renewcommand{\emptyset}{\varnothing}

\newcommand{\defn}[1]{\textcolor{Maroon}{\emph{#1}}}

\newcommand*{\cF}{\mathcal{F}}
\newcommand*{\cO}{\mathcal{O}}

\newcommand*{\cI}{\mathcal{I}}

\newcommand*{\cR}{\mathcal{R}}

\newcommand*{\cS}{\mathcal{S}}

\newcommand*{\cP}{\mathcal{P}}
\newcommand*{\cQ}{\mathcal{Q}}

\newcommand{\pos}[2]{i(#1,#2)}
\newcommand{\spos}[1]{i(#1)}

\newcommand{\slice}[2][]{S_{#1}(#2)}
\newcommand{\sliceperm}[2]{#1(#2)}

\DeclarePairedDelimiter{\set}{\{}{\}}
\DeclarePairedDelimiter{\abs}{\lvert}{\rvert}
\DeclarePairedDelimiter{\floor}{\lfloor}{\rfloor}
\DeclarePairedDelimiter{\ceil}{\lceil}{\rceil}

\title{Small families of partially shattering permutations}

\date{\today}

\author{
Ant\'onio Gir\~ao\footnotemark[2] \and Lukas Michel\footnotemark[2] \and Youri Tamitegama\footnotemark[2]
}

\begin{document}
\maketitle

\renewcommand{\thefootnote}{\fnsymbol{footnote}} 
\footnotetext[2]{Mathematical Institute, University of Oxford, United Kingdom
(\textsf{\{\href{mailto:girao@maths.ox.ac.uk}{girao},\href{mailto:michel@maths.ox.ac.uk}{michel},\href{mailto:tamitegama@maths.ox.ac.uk}{tamitegama}\}@maths.\allowbreak ox.ac.uk}).}

\renewcommand{\thefootnote}{\arabic{footnote}} 
\begin{abstract}
    We say that a family of permutations $t$-shatters a set if it induces at least $t$ distinct permutations on that set. What is the minimum number $f_k(n,t)$ of permutations of $\set{1, \dots, n}$ that $t$-shatter all subsets of size $k$? For $t \le 2$, $f_k(n,t) = \Theta(1)$. Spencer showed that $f_k(n,t) = \Theta(\log \log n)$ for $3 \le t \le k$ and $f_k(n,k!) = \Theta(\log n)$. In 1996, F\"uredi asked whether partial shattering with permutations must always fall into one of these three regimes. Johnson and Wickes recently settled the case $k = 3$ affirmatively and proved that $f_k(n,t) = \Theta(\log n)$ for $t > 2 (k-1)!$.
    
    We give a surprising negative answer to the question of F\"uredi by showing that a fourth regime exists for $k \ge 4$. We establish that $f_k(n,t) = \Theta(\sqrt{\log n})$ for certain values of $t$ and prove that this is the only other regime when $k = 4$. We also show that $f_k(n,t) = \Theta(\log n)$ for $t > 2^{k-1}$. This greatly narrows the range of $t$ for which the asymptotic behaviour of $f_k(n,t)$ is unknown.
\end{abstract}

\section{Introduction}

A family $\cP$ of permutations of $[n] = \set{1, \dots, n}$ \defn{shatters} a set $X \subseteq [n]$ if the permutations of $\cP$ induce every possible permutation on the elements of $X$. Shattering families of permutations were first studied by Spencer~\cite{spencer1972minimal} who asked the following question.
\begin{quote}
    What is the smallest family of permutations of $[n]$ that shatters all subsets of a fixed size $k$?
\end{quote}
Spencer~\cite{spencer1972minimal} showed that such families have size $\Theta(\log n)$, with subsequent work improving the constant of the lower bound \cite{ishigami1994extremal, furedi1996scrambling, radhakrishnan2003note}.

A natural refinement of this problem is to consider partial shattering. For $t \ge 1$, we say that a family $\cP$ \defn{$t$-shatters} a set $X$ if $\cP$ induces at least $t$ distinct permutations on $X$. Let $f_k(n,t)$ be the minimum number of permutations of $[n]$ that $t$-shatter all subsets of size~$k$. From above we know that $f_k(n,k!) = \Theta(\log n)$, and monotone permutations can be used to prove that $f_k(n,t) = t$ for $t \le 2$. Moreover, an argument of Hajnal and Spencer~\cite{spencer1972minimal} shows that $f_k(n,t) = \Theta(\log \log n)$ for $3 \le t \le k$.\footnote{Spencer~\cite{spencer1972minimal} considered the slightly different problem of determining the minimum number of permutations such that for every subset $X \subseteq [n]$ of size $k$ and every element $x \in X$, there is a permutation where $x$ is the largest element of $X$. Johnson and Wickes~\cite{johnson2023shattering} observed that this transfers to $f_k(n,t)$.} Therefore, the asymptotic behaviour of $f_k(n,t)$ falls into at least three distinct regimes.

In 1996, F\"uredi~\cite{furedi1996scrambling} asked whether these might be the only possible regimes, even in a much more general version of partial shattering. Let $\cS$ be a family of sets of permutations of~$[k]$. Then, a family $\cP$ of permutations of $[n]$ is \defn{$\cS$-mixing} if for every subset $X \subseteq [n]$ of size $k$, the set of permutations that $\cP$ induces on $X$ is a member of $\cS$. Moreover, $\cS$ is \defn{monotone} if for all $S \in \cS$ and $S \subseteq T$ we have $T \in \cS$. If $\cS$ is the family of sets with at least $t$ permutations of $[k]$, then $\cS$ is monotone and $\cS$-mixing families are exactly those families that $t$-shatter all subsets of size $k$.

Even in this more general $\cS$-mixing framework, the minimum size of an $\cS$-mixing family in all previously known cases was in one of the three regimes $\Theta(1)$, $\Theta(\log \log n)$, and $\Theta(\log n)$. This prompted F\"uredi~\cite{furedi1996scrambling} to ask the following question.

\begin{question}[F\"uredi, 1996]
    If $\cS$ is a monotone family of sets of permutations of $[k]$, is the minimum size of an $\cS$-mixing family either $\Theta(1)$, $\Theta(\log \log n)$, or $\Theta(\log n)$?
\end{question}

Johnson and Wickes~\cite{johnson2023shattering} recently made progress on this question for $f_k(n,t)$. They showed that $f_k(n,t) = \Theta(\log n)$ for $t > 2 (k-1)!$. Together with the previously known asymptotics, this yields the following partial classification.

\begin{theorem}[Johnson, Wickes, 2023]\label{thm:prevbounds}
    Let $k \ge 3$ be an integer. Then,
    \[
        f_k(n,t) = \begin{cases}
            t & \text{for } t \le 2, \\
            \Theta(\log \log n) & \text{for } 3 \le t \le k, \\
            \Theta(\log n) & \text{for } 2 (k-1)! < t \le k!.
        \end{cases}
    \]
\end{theorem}

Moreover, Johnson and Wickes~\cite{johnson2023shattering} settled the case $k = 3$ completely by additionally proving that $f_3(n,4) = \Theta(\log \log n)$. Given these results, they reiterated F\"uredi's question and asked specifically whether $f_k(n,t)$ must always fall into one of the three regimes $\Theta(1)$, $\Theta(\log \log n)$, and $\Theta(\log n)$ \cite{baber2023collection,johnson2023shattering}.

We answer the questions of F\"uredi and of Johnson and Wickes negatively. For $k \ge 4$, we show that a fourth regime exists with $f_k(n,t) = \Theta(\sqrt{\log n})$. More generally, we improve the partial classification of the asymptotic behaviour of $f_k(n,t)$ as follows.

\begin{theorem}\label{thm:classification}
    Let $k \ge 4$ be an integer. Then,
    \[
        f_k(n,t) = \begin{cases}
            t & \text{for } t \le 2, \\
            \Theta(\log \log n) & \text{for } 3 \le t \le 2^{\ceil{\log_2 k}}, \\
            \Theta(\sqrt{\log n}) & \text{for } 2^{\ceil{\log_2 k}} < t \le \min\set{2k, 2^{\ceil{\log_2 k}} + 4}, \\
            \Theta(\log n) & \text{for } 2^{k-1} < t \le k!.
        \end{cases}
    \]
\end{theorem}

For $k = 4$, this settles the asymptotic behaviour of $f_4(n,t)$ completely as all values of $t$ are covered. However, if $k$ is large, there is still an exponential gap between the regime $\Theta(\sqrt{\log n})$ and $\Theta(\log n)$.

\cref{thm:classification} is based on a series of new lower and upper bounds on $f_k(n,t)$. First, we show that the lower bound $f_k(n,t) = \Omega(\log n)$ holds for a wider range of values of $t$.

\begin{theorem}\label{thm:lblogn}
    Let $k \ge 3$ and $t > 2^{k-1}$. Then, $f_k(n,t) = \Omega(\log n)$.
\end{theorem}

The main observation for this result is that for any small family $\cP$ of permutations of~$[n]$, we can construct two large subsets $A, B \subseteq [n]$ that are ordered in the following sense: for each permutation of $\cP$ either all elements of $A$ are smaller than all elements of $B$, or all elements of $B$ are smaller than all elements of $A$. By recursively constructing ordered sets in $A$, we find a subset of size $k$ that is only $2^{k-1}$-shattered by $\cP$.

For smaller values of $t$, we provide a new lower bound of the form $\Omega(\sqrt{\log n})$.

\begin{theorem}\label{thm:lbsqrtlogn}
    Let $k \ge 3$ and $t > 2^{\ceil{\log_2 k}}$. Then, $f_k(n,t) = \Omega(\sqrt{\log n})$.
\end{theorem}

The proof of this result is inspired by the proof of \cref{thm:lblogn}. However, instead of only constructing ordered sets in $A$, we recursively construct ordered sets both in $A$ and in $B$. We then use a Ramsey-theoretic argument about vertex-coloured binary trees to find a subset of size $k$ that is only $2^{\ceil{\log k}}$-shattered.

We note that the lower bound on $t$ in \cref{thm:lbsqrtlogn} cannot be replaced by anything smaller. Indeed, a careful analysis of the construction of Hajnal and Spencer~\cite{spencer1972minimal} shows that $f_k(n,t) = \Theta(\log \log n)$ for $3 \le t \le 2^{\ceil{\log_2 k}}$. We provide an equivalent construction which proves this and which serves as a motivating example for what follows.

\begin{theorem}\label{thm:ubloglogn}
    Let $k \ge 3$ and $t \le 2^{\ceil{\log_2 k}}$. Then, $f_k(n,t) = \cO(\log \log n)$.
\end{theorem}

Our construction identifies $[n]$ with $[2]^d$ for $d = \log_2 n$. Then, we consider lexicographic permutations of $[2]^d$ which are premutations where the order of $x, y \in [2]^d$ only depends on the values $x_i$ and $y_i$ for the first position $i$ with $x_i \neq y_i$. We show that an appropriate choice of these permutations ensures that all subsets of size $k$ are $2^{\ceil{\log_2 k}}$-shattered.

Finally, for $k \ge 4$, we show that there exist $t > 2^{\ceil{\log_2 k}}$ with $f_k(n,t) = \cO(\sqrt{\log n})$. Together with \cref{thm:lbsqrtlogn}, this establishes the existence of a fourth regime for $f_k(n,t)$.

\begin{theorem}\label{thm:ubsqrtlogn}
    Let $k \ge 4$ and $t \le \min\set{2 k, 2^{\ceil{\log_2 k}} + 4}$. Then, $f_k(n,t) = \cO(\sqrt{\log n})$.
\end{theorem}

This result is proved similar to \cref{thm:ubloglogn}, but we identify $[n]$ with $[2^d]^d$ for $d = \sqrt{\log_2 n}$. Then, most subsets of size $k$ are $2k$-shattered by lexicographic permutations, and the remaining subsets have a very specific structure. We exploit this structure and add a few more permutations which ensure that all subsets of size $k$ are $t$-shattered.

We remark that partial shattering with permutations is quite different to partial shattering with sets. A family $\cF$ of subsets of $[n]$ \defn{shatters} a set $X \subseteq [n]$ if for every subset $Y \subseteq X$ there exists $F \in \cF$ with $F \cap X = Y$, and $\cF$ \defn{$t$-shatters} $X$ if for at least $t$ distinct subsets $Y \subseteq X$ there exists $F \in \cF$ with $F \cap X = Y$. The study of shattering families of sets dates back to the seminal works of Sauer~\cite{sauer1972density}, Shelah~\cite{shelah1972combinatorial}, and Vapnik and Chervonenkis~\cite{vapnik2015uniform}.

As for permutations, Kleitman and Spencer~\cite{kleitman1973} showed that the minimum number of subsets of $[n]$ that shatter all subsets of size $k$ is $\Theta(\log n)$. However, in the case of partial shattering, the family $\set{\emptyset, [n]}$ 2-shatters all subsets of size $k$, and every family that 3-shatters all subsets of size $k$ already needs $\Omega(\log n)$ sets.\footnote{If $n \ge k \cdot 2^{\abs{\cF}}$, one of the intersections $\bigcap_{F \in \cF} A_F$ with $A_F \in \set{F, [n] \setminus F}$ must have size at least $k$, and any subset $X \subseteq \bigcap_{F \in \cF} A_F$ of size $k$ satisfies $F \cap X \in \set{\emptyset, X}$ for all $F \in \cF$.} Therefore, partial shattering with sets only has the two regimes $\Theta(1)$ and $\Theta(\log n)$.

The rest of the paper is structured as follows. In \cref{sec:lowerbounds} we prove the lower bounds of \cref{thm:lblogn,thm:lbsqrtlogn}. Afterwards, in \cref{sec:upperbounds}, we prove the upper bounds of \cref{thm:ubloglogn,thm:ubsqrtlogn}. We finish with some open problems in \cref{sec:openproblems}.

\textbf{Notation.} Throughout the paper, a permutation $\rho$ of a set $X$ is a total order of the elements of $X$. We denote this order by $<_\rho$. Note that if $Y \subseteq X$, then $\rho$ induces a permutation on $Y$. If $\rho$ is a permutation of $A \cup B$, we write $A <_\rho B$ if for all $a \in A$ and $b \in B$ we have $a <_\rho b$.

\section{Lower bounds}
\label{sec:lowerbounds}

To prove lower bounds for $f_k(n,t)$, we first need to define the concept of ordered sets. Let $\cP$ be a family of permutations of a set $X$. Then, we say that a pair of disjoint subsets $A,B\subseteq X$ is \defn{$\cP$-ordered} if for each permutation $\rho \in \cP$ either $A <_\rho B$ or $B <_\rho A$. The following result shows that any set contains large $\cP$-ordered subsets.

\begin{lemma}\label{lem:pgoodpair}
    Let $X$ be a set and let $\cP$ be a family of $m$ permutations of $X$.
    Then, there exists a $\cP$-ordered pair $A,B\subseteq X$ with $\min\set{\abs{A}, \abs{B}} \ge \floor{\abs{X} / 2^{m+1}}$.
\end{lemma}

\begin{proof}
    Let $\cP = \set{\rho_1, \rho_2, \dotsc, \rho_m}$ and define $\cP_i = \set{\rho_1, \rho_2, \dots, \rho_i}$. We inductively construct a $\cP_i$-ordered pair $A_i,B_i \subseteq X$ with $\abs{A_i} = \abs{B_i} \ge \floor{\abs{X} / 2^{i+1}}$. For $i = 0$, we choose $A_0, B_0 \subseteq X$ as two disjoint subsets of size $\floor{\abs{X} / 2}$. 
    
    Now suppose we have constructed $A_i$ and $B_i$. Let $\ell = \abs{A_i} = \abs{B_i} \ge \floor{\abs{X} / 2^{i+1}}$ and consider the subset $L$ of the $\ell$ largest elements of $A_i \cup B_i$ under $<_{\rho_{i+1}}$. Note that either $\abs{L \cap A_i} \ge \ell / 2$ or $\abs{L \cap B_i} \ge \ell / 2$. Without loss of generality, assume that $\abs{L \cap A_i} \ge \ell / 2$.
    
    Set $A_{i+1} = L \cap A_i$ and $B_{i+1} = B_i \setminus L$. Since $b <_{\rho_{i+1}} a$ for all $a \in A_{i+1}$ and $b \in B_{i+1}$, the pair $(A_{i+1},B_{i+1})$ is $\cP_{i+1}$-ordered. We also have $\abs{A_{i+1}} \ge \ell / 2 \ge \floor{\abs{X} / 2^{i+2}}$ and $\abs{B_{i+1}} = \ell - \abs{L \cap B_i} = \ell - (\ell - \abs{L \cap A_i}) = \abs{A_{i+1}}$, as required.
    
    For $i = m$, the pair $(A_m, B_m)$ is $\cP$-ordered with $\min\set{\abs{A_m}, \abs{B_m}} \ge \floor{\abs{X} / 2^{m+1}}$.
\end{proof}

Using $\cP$-ordered pairs, we can prove \cref{thm:lblogn}. For this, we construct a sequence of nested $\cP$-ordered pairs $A_i, B_i \subseteq A_{i-1}$ for $i \in [k]$, where $A_0 = [n]$. Then, if we pick an element $x_i \in B_i$ for each $i \in [k]$, the fact that all pairs are $\cP$-ordered implies that $\set{x_1, \dots, x_k}$ is only $2^{k-1}$-shattered by $\cP$, as required.

\begin{proof}[Proof of \cref{thm:lblogn}]
    Let $t > 2^{k-1}$. We claim that $f_k(n,t) > (\log_2 n) / k - 1$. Indeed, suppose that $\cP$ is a family of $m$ permutations of $[n]$ with $m \le (\log_2 n) / k - 1$, and so $n \ge 2^{k (m+1)}$. We show that there is a subset of size $k$ that is not $t$-shattered by $\cP$.
    
    To do so, we inductively construct a sequence of $\cP$-ordered pairs as follows. Start with $A_0 = [n]$. Then, for $i = 1, \dots, k$, apply \cref{lem:pgoodpair} to $A_{i-1}$ to obtain a $\cP$-ordered pair $A_i, B_i \subseteq A_{i-1}$ with $\min\set{\abs{A_i}, \abs{B_i}} \ge \floor{\abs{A_{i-1}} / 2^{m+1}}$. Note that this implies that $\min\set{\abs{A_i}, \abs{B_i}} \ge \floor{n / 2^{i (m+1)}} \ge \floor{n / 2^{k (m+1)}} \ge 1$, and so $A_i$ and $B_i$ are non-empty.

    For each $i \in [k]$, pick $x_i \in B_i$. We claim that $\cP$ does not $t$-shatter the set $\set{x_1,\dotsc, x_k}$. Indeed, let $i \in [k-1]$. Note that $x_j \in A_i$ for all $j > i$. Since $(A_i,B_i)$ is $\cP$-ordered, it follows that for each permutation $\rho \in \cP$ we either have $x_i <_\rho \set{x_{i+1}, \dots, x_k}$ or $\set{x_{i+1}, \dots, x_k} <_\rho x_i$. This provides two choices for the position of $x_i$ relative to $\set{x_{i+1}, \dots, x_k}$. Moreover, given such a choice for each $i\in [k-1]$, this uniquely determines the permutation that $\rho$ induces on $\set{x_1,\dots,x_k}$. Hence, $\cP$ induces at most $2^{k-1}$ permutations on $\set{x_1,\dots,x_k}$, and therefore does not $t$-shatter $\set{x_1,\dots,x_k}$.
\end{proof}

To prove \cref{thm:lbsqrtlogn}, we will construct a complete binary tree of nested $\cP$-ordered pairs. Namely, for each $\cP$-ordered pair $(A,B)$ that we obtain, we recursively construct new $\cP$-ordered pairs in each of $A$ and $B$.

This is useful for the following reason. Suppose that $k = 4$, $(A,B)$ is a $\cP$-ordered pair, and two $\cP$-ordered pairs $A',B' \subseteq A$ and $A'',B'' \subseteq B$ in the tree are \defn{synchronised} in the sense that for all $\rho \in \cP$ we have $A' <_\rho B'$ if and only if $A'' <_\rho B''$. Pick $x_1 \in A'$, $x_2 \in B'$, $x_3 \in A''$, and $x_4 \in B''$. Then, since $(A,B)$ is $\cP$-ordered, we have $\set{x_1, x_2} <_\rho \set{x_3, x_4}$ or $\set{x_3, x_4} <_\rho \set{x_1, x_2}$ for all $\rho \in \cP$, and since $(A',B')$ and $(A'',B'')$ are synchronised, we have $x_1 <_\rho x_2$ if and only if $x_3 <_\rho x_4$. This implies that $\cP$ induces at most 4 permutations on $\set{x_1, x_2, x_3, x_4}$ which is less than $2^{k-1} = 8$.

To find synchronised $\cP$-ordered pairs, we use a Ramsey-theoretic argument about vertex-coloured binary trees. For that, we need to introduce some terminology. A \defn{binary tree} $T$ rooted at a vertex $r$ is a tree where $r$ has degree $0$ or $2$ and every other vertex has degree $1$ or $3$. A leaf of $T$ is a vertex of degree at most $1$. The \defn{layer} of $T$ at height $h$ is the set $N_h(r) \subseteq V(T)$ of those vertices at graph distance exactly $h$ from $r$. For a vertex $v \in N_h(r)$, its children are its neighbours in $N_{h+1}(r)$, and if $v \neq r$ then its parent is its neighbour in $N_{h-1}(r)$. Clearly, $\abs{N_h(r)} \le 2^h$. If there exists some $h$ such that $\abs{N_h(r)} = 2^h$ and $\abs{N_{h+1}(r)} = 0$, then $T$ is a \defn{complete binary tree} of height $h$.

Let $S$ be a second binary tree with root $q$. A \defn{subdivision} of $S$ in $T$ is an injective map $\varphi\colon V(S) \to V(T)$ such that for all vertices $v \in V(S) \setminus \set{q}$ with parent $p$, $\varphi(v)$ is contained in the subtree of $T$ rooted at $\varphi(p)$. If the vertices of $T$ are coloured (not necessarily properly), we say that the layers of the subdivision are \defn{monochromatic} if $\varphi(N_h(q))$ is monochromatic for all $h$. We often identify $V(S)$ with $\varphi(V(S))$.

In our proof, $T$ will be a complete binary tree whose vertices correspond to nested $\cP$-ordered pairs, and we colour the vertices of $T$ in such a way that monochromatic layers correspond to collections of synchronised $\cP$-ordered pairs. Then, if we find a subdivision of a complete binary tree of height $\ceil{\log_2 k}$ in $T$ with monochromatic layers, we can pick one element from each of its leaves to obtain a set of size $k$ that is only $2^{\ceil{\log_2 k}}$-shattered by $\cP$, as required.

To do this, we need to find such subdivisions in large vertex-coloured complete binary trees. For all integers $c \ge 1$ and $h \ge 0$, let $g(c,h)$ be the smallest integer such that every complete binary tree of height $g(c,h)$ whose vertices are coloured with $c$ colours contains a subdivision of a complete binary tree of height $h$ with monochromatic layers. We prove the following upper bound on $g(c,h)$.

\begin{lemma}\label{lem:monotree}
    For all integers $c \ge 1$ and $h \ge 0$ we have $g(c,h) \le h \ceil{\log_2(c^{h-1} + 1)}$.
\end{lemma}

\begin{proof}
    We proceed by induction on $h$. If $h = 0$, a complete binary tree of height $h$ is a single vertex, and so $g(c,0) = 0$.
    
    If $h \ge 1$, let $d = \ceil{\log_2(c^{h-1} + 1)}$ and let $T$ be a complete binary tree with root $r$ and height $d + g(c,h-1)$ that is coloured with $c$ colours. For each vertex $v \in N_d(r)$, the subtree of $T$ rooted at $v$ has height $g(c,h-1)$ and must therefore contain a subdivision $S_v$ of a complete binary tree of height $h-1$ with monochromatic layers. Since such subdivisions admit at most $c^{h-1}$ valid colourings and $\abs{N_d(r)} = 2^d \ge c^{h-1} + 1$, there must be two subdivisions $S_u$ and $S_v$ whose colourings coincide. Then, the subdivision formed by $S_u$, $S_v$, and the lowest common ancestor of $u$ and $v$ is a subdivision of a complete binary tree of height $h$ with monochromatic layers.

    Therefore, we get that
    \[
        g(c,h) \le d + g(c,h-1) \le d + (h-1) \ceil{\log_2(c^{h-2} + 1)} \le h \ceil{\log_2(c^{h-1} + 1)}. \qedhere
    \]
\end{proof}

Using the strategy outlined above, we now prove our lower bound for $t > 2^{\ceil{\log_2 k}}$.

\begin{proof}[Proof of \cref{thm:lbsqrtlogn}]
    Let $h = \ceil{\log_2 k}$ and $t > 2^h$. We claim $f_k(n,t) > \sqrt{\log_2 n} / h - 1$. Indeed, suppose that $\cP$ is a family of $m$ permutations of $[n]$ with $m \le \sqrt{\log_2 n} / h - 1$, and so using \cref{lem:monotree} we get $\log_2 n \ge h^2 (m+1)^2 \ge g(2^m, h) (m+1)$. We show that there is a subset of size $k$ that is not $t$-shattered by $\cP$.

    Consider a complete binary tree $T$ of height $g(2^m, h)$. We associate to each vertex $v$ of $T$ a set $X_v \subseteq [n]$ and a $\cP$-ordered pair $A_v, B_v \subseteq X_v$ as follows. For the root $r$ of $T$, let $X_r = [n]$ and let $A_r,B_r \subseteq X_r$ be a $\cP$-ordered pair given by \cref{lem:pgoodpair}. Then, for a vertex $v$ of $T$ with $\cP$-ordered pair $(A_v, B_v)$ and children $v_1$ and $v_2$, let $X_{v_1} = A_v$ and $X_{v_2} = B_v$, and let $A_{v_1}, B_{v_1} \subseteq X_{v_1}$ and $A_{v_2}, B_{v_2} \subseteq X_{v_2}$ be $\cP$-ordered pairs given by \cref{lem:pgoodpair}. This implies that $\abs{X_w} \ge \floor{n / 2^{g(2^m, h) (m+1)}} \ge 1$ for every $w \in V(T)$, and so $X_w$ is non-empty.
    
    Colour each vertex $v \in V(T)$ with a colour $c_v$ as follows. For $\rho \in \cP$, let
    \[
        c_v(\rho) = \begin{cases}
            1 & \text{if } A_v <_\rho B_v, \\
            0 & \text{if } B_v <_\rho A_v.
        \end{cases}
    \]
    Since $(A_v,B_v)$ is $\cP$-ordered, this is well-defined. This colours each vertex of $T$ with one of $2^m$ colours. By \cref{lem:monotree}, $T$ admits a subdivision of a complete binary tree $S$ of height $h$ with monochromatic layers. Let $q$ be the root of $S$.
    
    For each leaf $\ell$ of $S$, pick $y_\ell \in X_\ell$, and let $Y = \set{y_\ell : \ell \in N_h^S(q)}$ be the set of all of these elements. Note that $\abs{Y} = 2^h \ge k$. We claim that $Y$ is not $t$-shattered. Indeed, consider the colouring that a fixed permutation $\rho \in \cP$ induces on the first $h$ layers $N_{< h}^S(q)$ of $S$ by assigning $c_v(\rho)$ to each $v\in N_{< h}^S(q)$. Since each layer of $S$ is monochromatic, this colours $N_{< h}^S(q)$ with one of at most $2^h$ distinct colourings (accross all $\rho \in \cP$).
    
    Now, for any two distinct leaves $\ell, \ell' \in N_h^S(q)$, the lowest common ancestor $v \in N_{< h}^S(q)$ of $\ell$ and $\ell'$ satisfies $y_{\ell} \in A_v$ and $y_{\ell'} \in B_v$ (or vice versa). Then, if $c_v(\rho) = 1$, we have $A_v <_\rho B_v$ and so $y_{\ell} <_\rho y_{\ell'}$, and if $c_v(\rho) = 0$ we have $B_v <_\rho A_v$ and so $y_{\ell'} <_\rho y_{\ell}$. Thus, the colours $c_v(\rho)$ for the vertices $v \in N_{< h}^S(q)$ uniquely determine the permutation that $\rho$ induces on $Y$. Since there are at most $2^h$ distinct colourings of $N_{< h}^S(q)$, it follows that $\cP$ induces at most $2^h$ permutations on $Y$, and therefore does not $t$-shatter $Y$.
\end{proof}

\section{Upper bounds}
\label{sec:upperbounds}

To prove our upper bounds on $f_k(n,t)$, we use lexicographic permutations. Identify $[n]$ with a subset of $[b]^d$ where $b \ge 2$ and $d \ge \ceil{\log_b n}$ and define a \defn{lex-permutation} $\rho$ of $[b]^d$ to be a permutation that is constructed as follows. For each $i \in [d]$, pick a permutation $\rho_i$ of $[b]$. Then, for distinct $x, y \in [b]^d$, set $x <_\rho y$ if and only if $x_i <_{\rho_i} y_i$ where $i = \pos{x}{y} = \min\set{i \in [d] : x_i \neq y_i}$. That is, the relative order of $x$ and $y$ in $\rho$ is entirely determined by the first position where $x$ and $y$ differ. It is easy to check that $<_\rho$ is a total order, and so $\rho$ is a permutation.

We say that a family of permutations $\cP$ of $[b]^d$ is \defn{$k$-lex-shattering} if it is a family of lex-permutations such that the following holds. Let $\cI \subseteq [d]$ be a subset of size at most $k$, and for each $i \in \cI$ fix a permutation $\sigma_i$ of a subset $Y_i \subseteq [b]$ of size at most $k$. Then, there exists a permutation $\rho \in \cP$ such that $\rho_i$ induces $\sigma_i$ on $Y_i$ for all $i \in \cI$. We will later show that $k$-lex-shattering families $2^{\ceil{\log_2 k}}$-shatter every subset of size $k$. However, we first have to show that small $k$-lex-shattering families exist.

\begin{lemma}\label{lem:lexshattersize}
    For $k \ge 1$, there exists a $k$-lex-shattering family of $[b]^d$ with size $\cO(\log(b d))$.
\end{lemma}

\begin{proof}
    By \cref{thm:prevbounds}, there is a family $\cQ$ of permutations of $[b d]$ with size $\cO(\log(b d))$ that shatters every subset of size $k^2$. For each $\tau \in \cQ$, construct a lex-permutation $\rho^\tau$ such that $\rho^\tau_i$ is the permutation that $\tau$ induces on $(i-1) b + [b] = \set{(i-1) b + 1, \dotsc, i b}$. We claim that $\cP = \set{\rho^\tau : \tau \in \cQ}$ is $k$-lex-shattering.

    Indeed, let $\cI \subseteq [d]$ have size at most $k$, and for every $i \in \cI$ let $\sigma_i$ be a permutation of a subset $Y_i \subseteq [b]$ of size at most $k$. Since $\cQ$ shatters the set $\bigcup_{i \in \cI} ((i-1) b + Y_i)$, there exists a permutation $\tau \in \cQ$ which induces $\sigma_i$ on $(i-1) b + Y_i$ for all $i \in \cI$. This implies that $\rho^\tau_i$ induces $\sigma_i$ on $Y_i$ for all $i \in \cI$.
\end{proof}

To determine the permutations that a $k$-lex-shattering family $\cP$ induces on a subset $X$, consider the lexicographic structure of $X$. Let $\spos{X} = \inf\set{\pos{x}{y} : x, y \in X}$ be the first position where two elements of $X$ differ, and let $\slice{X} = \set{x_{\spos{X}} : x \in X}$ denote the \defn{slice} of $X$ at that position if $\spos{X} < \infty$. This partitions $X$ into the subsets $X_s = \set{x \in X : x_{\spos{X}} = s}$ for $s \in \slice{X}$. Note that the sets $X_s$ are pairwise $\cP$-ordered and that their relative order in a permutation $\rho \in \cP$ is determined by the permutation that $\rho_{\spos{X}}$ induces on $\slice{X}$. We will denote this permutation by $\sliceperm{\rho}{X}$.

If we partition a set $X_s$ in the same way, this yields a partition of $X_s$ into $\cP$-ordered sets whose relative order is determined by $\sliceperm{\rho}{X_s}$. Continuing recursively like this, it follows that the order that $\rho$ induces on $X$ is determined by the permutations $\sliceperm{\rho}{Y}$ for $Y \subseteq X$. This is captured by the following lemma.

\begin{lemma}\label{lem:lexperminduced}
    Let $\rho$ be a lex-permutation of $[b]^d$ and let $X \subseteq [b]^d$. Then, the permutation that $\rho$ induces on $X$ is uniquely determined by the set of permutations $\sliceperm{\rho}{Y}$ for $Y \subseteq X$.
\end{lemma}

\begin{proof}
    Let $x, y \in X$ and $i = \pos{x}{y}$. Then, $Y = \set{x, y} \subseteq X$ and $\sliceperm{\rho}{Y}$ is the permutation that $\rho_i$ induces on $\slice{Y} = \set{x_i, y_i}$. This implies that $x <_\rho y$ if and only if $x_i <_{\rho_i} y_i$ which in turn is equivalent to $x_i <_{\sliceperm{\rho}{Y}} y_i$. So, the permutation that $\rho$ induces on $X$ is determined by the set of permutations $\sliceperm{\rho}{Y}$ for $Y \subseteq X$.

    Conversely, suppose that $\rho$ and $\tau$ are lex-permutations with $\sliceperm{\rho}{Y} \neq \sliceperm{\tau}{Y}$ for some $Y \subseteq X$. Let $i = \spos{Y}$. Since $\sliceperm{\rho}{Y} \neq \sliceperm{\tau}{Y}$, there must exist $x, y \in Y$ with $\pos{x}{y} = i$ such that $x_i <_{\sliceperm{\rho}{Y}} y_i$ and $y_i <_{\sliceperm{\tau}{Y}} x_i$ (or vice versa). As above, this is equivalent to $x <_\rho y$ and $y <_\tau x$, and so $\rho$ and $\tau$ induce different permutations on $X$.
\end{proof}

Next, we deduce a lower bound on the number of permutations that $\cP$ induces on~$X$.
Let $\cI(X) = \set{\spos{Y} : Y \subseteq X \text{ and } \spos{Y} < \infty}$. We rely on the following oberservation.

\begin{lemma}\label{lem:product}
    Let $\cP$ be a $k$-lex-shattering family of permutations of $[b]^d$ and let $X \subseteq [b]^d$ be a subset of size $k$. For each $i \in \cI(X)$, let $Y_i \subseteq X$ be such that $\spos{Y_i} = i$. Then, $X$ is $(\prod_{i\in \cI(X)}\abs{S(Y_i)}!)$-shattered by $\cP$.
\end{lemma}

\begin{proof}
    We claim that $\abs{\cI(X)} < k$. Indeed, write $<$ for the standard lexicographic order on $[b]^d$, so $x < y$ if and only if $x_i < y_i$ where $i = \pos{x}{y}$. Note that if $x, y, z \in [b]^d$ satisfy $x < y < z$, then $i(x,z) = \min\set{i(x,y), i(y,z)}$. Therefore, if we write $X = \set{x_1, \dots, x_k}$ with $x_1 < \dots < x_k$, this implies that $\cI(X) = \set{i(x_\ell, x_{\ell+1}) : \ell \in [k-1]}$, and so $\abs{\cI(X)} \le k-1$ as claimed.

    Since $\cP$ is $k$-lex-shattering, it follows that if $\sigma_i$ is a permutation of $\slice{Y_i}$ for every $i \in \cI(X)$, then there is a permutation $\rho\in\cP$ with $\sliceperm{\rho}{Y_i} = \sigma_i$ for all $i \in \cI(X)$. By \cref{lem:lexperminduced}, each choice of the permutations $\sigma_i$ induces a different permutation on $X$, and so $X$ is $(\prod_{i \in \cI(X)} \abs{\slice{Y_i}}!)$-shattered by $\cP$.
\end{proof}

To maximise this product, consider again the decomposition of $X$ into its lexicographic structure from above. Suppose that we only pick the largest set $X_s$ whenever we recursively continue the decomposition, and let the sets obtained during this process be the sets $Y_i$. Then, we will show that $\prod_{i \in \cI} \abs{\slice{Y_i}}! \ge 2^{\ceil{\log_2 k}}$, and so every set of size $k$ is $2^{\ceil{\log_2 k}}$-shattered by $\cP$. This suffices to prove \cref{thm:ubloglogn}.

For \cref{thm:ubsqrtlogn}, we will additionally infer some structural information about $X$ whenever $X$ is not $2 k$-shattered by $\cP$. Later, this information will allow us to add more permutations to $\cP$ which ensure that $X$ is $t$-shattered for some $t > 2^{\ceil{\log_2 k}}$. Overall, we obtain the following lemma.

\begin{lemma}\label{lem:lexshatter2h}
    Let $\cP$ be a $k$-lex-shattering family of permutations of $[b]^d$, and let $h = \ceil{\log_2 k}$. Then, $\cP$ is a family that $2^h$-shatters all subsets of $[b]^d$ of size $k$. Moreover, if $X \subseteq [b]^d$ is a subset of size $k$ that $\cP$ does not $2k$-shatter, then the set $\cI(X)$ has size $h$, all $Y \subseteq X$ with $\spos{Y} < \infty$ satisfy $\abs{\slice{Y}} = 2$, and all $Y, Z \subseteq X$ with $\spos{Y} = \spos{Z} < \infty$ satisfy $\slice{Y} = \slice{Z}$.
\end{lemma}

\begin{proof}
    Let $X \subseteq [b]^d$ be a subset of size $k$ and let $\cI = \cI(X)$.

    \begin{claim}\label{clm:largeyi}
        There exist sets $Y_i \subseteq X$ with $\spos{Y_i} = i$ for each $i \in \cI$ such that $\prod_{i \in \cI} \abs{\slice{Y_i}} \ge k$.
    \end{claim}

    \begin{poc}
        We prove by induction on $j$ that for all $0 \le j \le d$ there exist sets $Y_i \subseteq X$ with $\spos{Y_i} = i$ for each $i \in \cI \cap [j]$ and a subset $Z \subseteq X$ which satisfy $\spos{Z} > j$ and $\abs{Z} \cdot \prod_{i \in \cI \cap [j]} \abs{\slice{Y_i}} \ge k$. Applying this with $j = d$ proves the claim.
        
        For $j = 0$, let $Z = X$. Now suppose we have constructed such sets for $j - 1$. If $\spos{Z} > j$, the same sets work for $j$, where we choose any $Y_j\subseteq X$ with $i(Y_j)=j$ if $j\in \cI$.

        Otherwise, $\spos{Z} = j$. Let $Y_j = Z$, and for every $s \in \slice{Z}$ let $Z_s = \set{z \in Z : z_j = s}$. Pick $s \in \slice{Z}$ such that $\abs{Z_s} \ge \abs{Z} / \abs{\slice{Z}}$. Note that $\spos{Z_s} > j$ and $\abs{Z_s} \cdot \prod_{i \in \cI \cap [j]} \abs{\slice{Y_i}} \ge (\abs{Z} / \abs{\slice{Z}}) \cdot \prod_{i \in \cI \cap [j-1]} \abs{\slice{Y_i}} \cdot \abs{\slice{Z}} \ge k$. So we have constructed the sets for $j$.
    \end{poc}

    For all $i \in \cI$, let $Y_i \subseteq X$ be a set with $\spos{Y_i} = i$ that maximises $\abs{\slice{Y_i}}$. 
    By \cref{clm:largeyi}, $\prod_{i \in \cI} \abs{\slice{Y_i}} \ge k$, and by \cref{lem:product} $\cP$ induces at least $\prod_{i \in \cI} \abs{\slice{Y_i}}!$ permutations on $X$. If $\abs{\slice{Y_i}} \ge 3$ for some $i \in \cI$, then $\prod_{i \in \cI} \abs{\slice{Y_i}}! \ge 2 \prod_{i \in \cI} \abs{\slice{Y_i}} \ge 2 k \ge 2^h$. Otherwise, $\abs{\slice{Y_i}} = 2$ for all $i \in \cI$. Then, since $2^{\abs{\cI}} = \prod_{i \in \cI} \abs{\slice{Y_i}} \ge k$, we have $\abs{\cI} \ge h$ and so $\prod_{i \in \cI} \abs{\slice{Y_i}}! = 2^{\abs{\cI}} \ge 2^h$. This shows that $X$ is $2^h$-shattered.

    Moreover, if $X$ is not $2k$-shattered, these arguments imply that $\abs{\slice{Y_i}} = 2$ for all $i \in \cI$ and $\abs{\cI} = h$. Suppose that for some $j \in \cI$ there exists a set $Z \subseteq X$ with $\spos{Z} = j$ and $\slice{Y_j} \neq \slice{Z}$. Then, by the maximality of $Y_j$, $\abs{\slice{Z}} = 2$ and so $\abs{\slice{Y_j} \cap \slice{Z}} \le 1$. In particular, for any pair of permutations $\sigma$ of $\slice{Y_j}$ and $\tau$ of $\slice{Z}$, there is a permutation $\sigma_j$ of $\slice{Y_j} \cup \slice{Z}$ that induces $\sigma$ on $\slice{Y_j}$ and $\tau$ on $\slice{Z}$. By the same argument as in the proof of \cref{lem:product}, it follows that $\cP$ induces at least $\abs{\slice{Z}}!\cdot\prod_{i \in \cI} \abs{\slice{Y_i}}! \ge 2 k$ permutations on $X$. This contradicts the fact that $X$ is not $2k$-shattered.
\end{proof}

With $b = 2$ and $d = \ceil{\log_2 n}$, \cref{thm:ubloglogn} is an immediate consequence of \cref{lem:lexshattersize,lem:lexshatter2h}. We remark that while the description of our construction differs significantly from that of Spencer \cite{spencer1972minimal}, these constructions are essentially equivalent.

To prove \cref{thm:ubsqrtlogn}, start with a $k$-lex-shattering family $\cP$. By \cref{lem:lexshatter2h}, most subsets of size $k$ are $2k$-shattered by $\cP$, and even if a subset $X$ is not $2k$-shattered, it is nevertheless $2^{\ceil{\log_2 k}}$-shattered and it has a very specific structure. We exploit this structure by adding new permutations to $\cP$ which induce at least four additional permutation on $X$ and thereby ensure that $X$ is $(2^{\ceil{\log_2 k}} + 4)$-shattered.

To do so, we add a constant number of new permutations for each position $i \in [d]$. Let $<$ and $>$ denote the standard and reverse permutations of $[b]$.\footnote{That is, for $x,y\in [b]$ set $x <_< y$ if and only if $x < y$, and $x <_> y$ if and only if $x > y$.} Then, for all $\sigma, \tau \in \set{<,>}$, let $\pi_{i,\sigma,\tau}$ be the permutation of $[b]^d$ with $x <_{\pi_{i,\sigma,\tau}} y$ if and only if either $x_i <_\sigma y_i$, or $x_i = y_i$ and $x_{\pos{x}{y}} <_\tau y_{\pos{x}{y}}$. That is, $\pi_{i,\sigma,\tau}$ first sorts according to position $i$ and only afterwards behaves like a lex-permutation. Define
\[
    \cQ_i = \set{\pi_{i,\sigma,\tau} : \sigma, \tau \in \set{<,>}}.
\]
If $X$ is not $2k$-shattered, we show that for an appropriate $i \in [d]$ the permutations of $\cQ_i$ induce four additional permutations on $X$, as required.

\begin{proof}[Proof of \cref{thm:ubsqrtlogn}]
    Let $b = 2^d$, $d = \ceil{\sqrt{\log_2 n}}$, and $h = \ceil{\log_2 k}$. By \cref{lem:lexshattersize}, there exists a $k$-lex-shattering family $\cP$ of $[b]^d$ with size $\cO(\sqrt{\log n})$. Define $\cR = \cP \cup \bigcup_{i \in [d]} \cQ_i$ and note that $\cR$ has size $\cO(\sqrt{\log n})$.
    
    We claim that every subset $X \subseteq [b]^d$ of size $k$ is $\min\set{2k,2^h+4}$-shattered by $\cR$. Indeed, suppose that $X$ is not $2k$-shattered by $\cP$. Then, by \cref{lem:lexshatter2h}, we know that the set $\cI(X)$ has size $h$, all $Y \subseteq X$ with $\spos{Y} < \infty$ satisfy $\abs{\slice{Y}} = 2$, and all $Y, Z \subseteq X$ with $\spos{Y} = \spos{Z} < \infty$ satisfy $\slice{Y} = \slice{Z}$.

    Write $X = \set{x_1, \dots, x_k}$ with $x_1 < \dots < x_k$. Note that $\pos{x_\ell}{x_{\ell+1}} \in \cI(X)$ for all $\ell \in [k-1]$. Since $k \ge 4$, one can check that $k-1 > h = \abs{\cI(X)}$. So there must exist $1 \le \ell < m < k$ with $\pos{x_\ell}{x_{\ell+1}} = \pos{x_m}{x_{m+1}}$. Let $i = \pos{x_\ell}{x_{\ell+1}}$ and $j = \pos{x_{\ell+1}}{x_m}$.

    If $j \ge i$, then $\pos{x_{\ell+1}}{x_{m+1}} = i$ and $\pos{x_\ell}{x_{m+1}} = i$. So, $Y = \set{x_\ell, x_{\ell+1}, x_{m+1}} \subseteq X$ satisfies $\abs{\slice{Y}} = 3$. This contradicts the fact that all $Y \subseteq X$ with $\spos{Y} < \infty$ satisfy $\abs{\slice{Y}} = 2$. Therefore, $j < i$. In particular, $(x_\ell)_j = (x_{\ell+1})_j < (x_m)_j = (x_{m+1})_j$, and so $(\set{x_\ell, x_{\ell+1}}, \set{x_m, x_{m+1}})$ is $\cP$-ordered.

    Note that $\spos{\set{x_\ell, x_{\ell+1}}} = i = \spos{\set{x_m, x_{m+1}}}$. Since all $Y, Z \subseteq X$ with $\spos{Y} = \spos{Z} < \infty$ satisfy $\slice{Y} = \slice{Z}$, we must therefore have $(x_\ell)_i = (x_m)_i < (x_{\ell+1})_i = (x_{m+1})_i$, and so $(\set{x_\ell, x_m}, \set{x_{\ell+1}, x_{m+1}})$ is $\cQ_i$-ordered.

    In particular, $\cP$ and $\cQ_i$ induce different permutations on $X$. Moreover, note that $\set{x_\ell, x_{\ell+1}, x_m, x_{m+1}}$ is $4$-shattered by $\cQ_i$. Since \cref{lem:lexshatter2h} implies that $X$ is $2^h$-shattered by $\cP$, it follows that $X$ is $(2^h+4)$-shattered by $\cR$.
\end{proof}

\section{Open Problems}
\label{sec:openproblems}

In this paper, we have shown that at least four regimes exist for the asymptotic behaviour of $f_k(n,t)$ when $k \ge 4$, and we narrowed the range of values of $t$ for which the asymptotic behaviour is unknown. The main open problem is to determine the asymptotic behaviour of $f_k(n,t)$ for $k \ge 5$ and $\min\set{2k, 2^{\ceil{\log_2 k}} + 4} < t \le 2^{k-1}$. The following remains possible.

\begin{question}
    For all integers $k$ and $t$, is the asymptotic behaviour of $f_k(n,t)$ either $\Theta(1)$, $\Theta(\log \log n)$, $\Theta(\sqrt{\log n})$, or $\Theta(\log n)$?
\end{question}

We conjecture that \cref{thm:lblogn} determines the entire range of values of $t$ that satisfy $f_k(n,t) = \Theta(\log n)$.

\begin{conjecture}\label{conj:upologn}
    Let $k \ge 3$ and $t \le 2^{k-1}$. Then, $f_k(n,t) = o(\log n)$.
\end{conjecture}

Since we did not try to optimise the upper bound on $t$ in \cref{thm:ubsqrtlogn}, we believe that constructions similar to ours could prove that $f_k(n,t) = o(\log n)$ for a larger range of values of $t$. However, they seem to be far away from reaching $t = 2^{k-1}$.

In this context it seems important to mention the following inspiration for our approach that comes from the Erd\H{o}s-Gy\'{a}rf\'{a}s problem~\cite{erdos1997AVO}. If $\cP$ is a family of permutations of $[n]$, assign a colour $c_{x,y}$ to every pair $x, y \in [n]$ as follows. For $\rho \in \cP$, let
\[
    c_{x,y}(\rho) = \begin{cases}
        1 & \text{if } x <_\rho y, \\
        0 & \text{otherwise}.
    \end{cases}
\]
This colouring uses $2^{\abs{\cP}}$ colours, and if a subset $X \subseteq [n]$ of size $k$ spans at most $\ell$ colours, then $\cP$ induces at most $2^\ell$ permutations on $X$. We implicitely used this approach to prove \cref{thm:lblogn,thm:lbsqrtlogn}. Unfortunately, the converse does not hold: $X$ can simultaneously span at least $\ell$ colours and fail to be $2^\ell$-shattered by $\cP$.

Nevertheless, Eichhorn and Mubayi~\cite{eichhorn2000note} gave a colouring of $K_n$ with $2^{\Theta(\sqrt{\log n})}$ colours such that all subsets of size $k$ span at least $\ell = 2 \ceil{\log_2 k} - 2$ colours. This inspired our construction for \cref{thm:ubsqrtlogn}, even if our construction does not $2^\ell$-shatter all subsets of size $k$. Note that Conlon, Fox, Lee, and Sudakov~\cite{conlon2015erdos} gave a colouring of $K_n$ with $2^{o(\log n)}$ colours such that all subsets of size $k$ span at least $k - 1$ colours. We hope that these colourings could inspire further progress towards better upper bounds on $f_k(n,t)$.

\bibliographystyle{style_edited}
\bibliography{references}
\end{document}